\documentclass[a4paper,11pt,oneside]{article}
\usepackage{amsmath,titlesec,amsthm,amssymb}

\titleformat{\section}{\large\bf\boldmath}{\arabic{section}.}{2ex}{}

\titlespacing*{\section}{0ex}{2ex}{1ex}

\titleformat{\subsection}{\bf\boldmath}{\arabic{section}.\arabic{subsection}.}{2ex}{}

\titlespacing*{\subsection}{0ex}{0.8ex}{0ex}

\newtheorem{thmstar}{Theorem}
\newtheorem{theorem}{Theorem}
\newtheorem{lemma}[theorem]{Lemma}

\newtheorem{corollary}[theorem]{Corollary}

{\theoremstyle{definition}

\newtheorem{remark}[theorem]{Remark}
}

\newcommand{\recht}{\rightarrow}

\newcommand{\om}{\omega}

\newcommand{\vphi}{\varphi}
\newcommand{\cF}{\mathcal{F}}
\newcommand{\cC}{\mathcal{C}}
\newcommand{\cI}{\mathcal{I}}
\newcommand{\cG}{\mathcal{G}}
\newcommand{\rL}{\mathord{\text{\rm L}}}
\newcommand{\Om}{\Omega}
\newcommand{\N}{\mathbb{N}}
\newcommand{\cU}{\mathcal{U}}

\newcommand{\cJ}{\mathcal{J}}
\newcommand{\id}{\mathord{\text{\rm id}}}
\newcommand{\Z}{\mathbb{Z}}
\newcommand{\actson}{\curvearrowright}

\newcommand{\al}{\alpha}
\newcommand{\R}{\mathbb{R}}

\newcommand{\eps}{\varepsilon}

\newcommand{\F}{\mathbb{F}}
\newcommand{\si}{\sigma}

\newcommand{\cB}{\mathcal{B}}

\newcommand{\cV}{\mathcal{V}}

\newcommand{\Xtil}{\widetilde{X}}
\newcommand{\Ytil}{\widetilde{Y}}
\newcommand{\yb}{\overline{y}}
\newcommand{\xb}{\overline{x}}

\newcommand{\rhob}{\overline{\rho}}
\newcommand{\Atil}{\tilde{A}}
\newcommand{\Btil}{\tilde{B}}
\newcommand{\Stab}{\operatorname{Stab}}

\begin{document}
\begin{center}
{\LARGE\bf\boldmath Stable orbit equivalence of Bernoulli actions of free\vspace{0.3ex}\\ groups and isomorphism of some of their factor actions}

\vspace{1ex}

by Niels Meesschaert, Sven Raum and Stefaan Vaes\renewcommand{\thefootnote}{}\footnotetext{\noindent Department of Mathematics, K.U.Leuven.\\ Work partially
    supported by ERC Starting Grant VNALG-200749, Research
    Programme G.0639.11 of the Research Foundation --
    Flanders (FWO) and K.U.Leuven BOF research grant OT/08/032.}
\end{center}

\begin{abstract}\noindent
We give an elementary proof for Lewis Bowen's theorem saying that two Bernoulli actions of two free groups, each having arbitrary base probability spaces, are stably orbit equivalent. Our methods also show that for all compact groups $K$ and every free product $\Gamma$ of infinite amenable groups,
the factor $\Gamma \actson K^{\Gamma}/K$ of the Bernoulli action $\Gamma \actson K^{\Gamma}$ by the diagonal $K$-action, is isomorphic with a Bernoulli action of $\Gamma$.
\end{abstract}

Free, ergodic and probability measure preserving (p.m.p.) actions $\Gamma \actson (X,\mu)$ of countable groups give rise to II$_1$ factors $\rL^\infty(X) \rtimes \Gamma$ through the group measure space construction of Murray and von Neumann. It was shown in \cite{Si55} that the isomorphism class of the II$_1$ factor $\rL^\infty(X) \rtimes \Gamma$ only depends on the orbit equivalence relation on $(X,\mu)$ given by $\Gamma \actson (X,\mu)$. This led Dye in \cite{Dy58} to a systematic study of group actions up to orbit equivalence, where he proved the fundamental result that all free ergodic p.m.p.\ actions of $\Z$ are orbit equivalent. Note that two such actions need not be isomorphic (using entropy, spectral measure, etc). In \cite{OW79} Ornstein and Weiss showed that actually all orbit equivalence relations of all free ergodic p.m.p.\ actions of infinite amenable groups are isomorphic with the unique ergodic hyperfinite equivalence relation of type II$_1$.

The nonamenable case is far more complex and many striking rigidity results have been established over the last 20 years, leading to classes of group actions for which the orbit equivalence relation entirely determines the group and its action. We refer to \cite{Sh05,Fu09,Ga10} for a comprehensive overview of measured group theory. On the other hand there have so far only been relatively few orbit equivalence ``flexibility'' results for nonamenable groups. Two results of this kind have been obtained recently by Lewis Bowen in \cite{Bo09a,Bo09b}. In \cite{Bo09a} Bowen proved that two Bernoulli actions $\F_n \actson X_0^{\F_n}$ and $\F_n \actson X_1^{\F_n}$ of the same free group $\F_n$, but with different base probability spaces, are always orbit equivalent. Note that this is a nontrivial result because Bowen proved earlier in \cite{Bo08} that these Bernoulli actions can only be isomorphic if the base probability spaces $(X_0,\mu_0)$ and $(X_1,\mu_1)$ have the same entropy.

Two free ergodic p.m.p.\ actions $\Gamma_i \actson (X_i,\mu_i)$ are called stably orbit equivalent if their orbit equivalence relations can be restricted to non-negligible measurable subsets $\cU_i \subset X_i$ such that the resulting equivalence relations on $\cU_0$ and $\cU_1$ become isomorphic. The number $\mu_1(\cU_1)/\mu_0(\cU_0)$ is called the compression constant of the stable orbit equivalence. In \cite{Bo09b} Bowen proved that the Bernoulli actions $\F_n \actson X_0^{\F_n}$ and $\F_m \actson X_1^{\F_m}$ of two different free groups are stably orbit equivalent with compression constant $(n-1)/(m-1)$.

The first aim of this article is to give an elementary proof for the above two theorems of Bowen. The concrete stable orbit equivalence that we obtain between $\F_n \actson X_0^{\F_n}$ and $\F_m \actson X_1^{\F_m}$ is identical to the one discovered by Bowen. The difference between the two approaches is however the following: rather than writing an explicit formula for the stable orbit equivalence, we construct actions of $\F_n$ and $\F_m$ on (subsets of) the same space, having the same orbits and satisfying an abstract characterization of the Bernoulli action.

Secondly our simpler methods also yield a new orbit equivalence flexibility (actually isomorphism) result that we explain now. Combining the work of many hands \cite{GP03,Io06,GL07} it was shown in \cite{Ep07} that every nonamenable group admits uncountably many non orbit equivalent actions (see \cite{Ho11} for a survey). Nevertheless it is still an open problem to give a concrete construction producing such an uncountable family. For a while it has been speculated that for any given nonamenable group $\Gamma$ the actions
\begin{equation}\label{eq.family}
\Bigl\{\Gamma \actson K^\Gamma / K \; \Big| \; \text{$K$ a compact second countable group acting by diagonal translation on $K^\Gamma$}\; \Bigr\}
\end{equation}
are non orbit equivalent for nonisomorphic $K$. Indeed, in \cite[Proposition 5.6]{PV06} it was shown that this is indeed the case whenever every $1$-cocycle for the Bernoulli action $\Gamma \actson K^\Gamma$ with values in either a countable or a compact group $\cG$ is cohomologous to a group homomorphism from $\Gamma$ to $\cG$. By Popa's cocycle superrigidity theorems \cite{Po05,Po06}, this is the case when $\Gamma$ contains an infinite normal subgroup with the relative property (T) or when $\Gamma$ can be written as the direct product of an infinite group and a nonamenable group. Conjecturally the same is true whenever the first $\ell^2$-Betti number of $\Gamma$ vanishes (cf.\ \cite{PS09}).

In the last section of this paper we disprove the above speculation whenever $\Gamma = \Lambda_1 * \cdots * \Lambda_n$ is the free product of $n$ infinite amenable groups, in particular when $\Gamma = \F_n$. We prove that for these $\Gamma$ and for every compact second countable group $K$ the action $\Gamma \actson K^\Gamma/K$ is isomorphic with a Bernoulli action of $\Gamma$. As we shall see, the special case $\Gamma = \F_n$ is a very easy generalization of \cite[Appendix C.(b)]{OW86} where the same result is proven for $K = \Z / 2 \Z$ and $\Gamma = \F_2$.

More generally, denote by $\cG$ the class of countably infinite groups $\Gamma$ for which the action $\Gamma \actson K^{\Gamma}/K$ is isomorphic with a Bernoulli action of $\Gamma$. Then by \cite{OW86} the class $\cG$ contains all infinite amenable groups. We prove in Theorem \ref{thm.stability} that $\cG$ is stable under taking free products. By the results cited above, $\cG$ does not contain groups that admit an infinite normal subgroup with the relative property (T) and $\cG$ does not contain groups that can be written as the direct product of an infinite group and a nonamenable group. So it is a very intriguing problem which groups belong to $\cG$.

\subsection*{Terminology and notations}

A measure preserving action $\Gamma \actson (X,\mu)$ of a countable group $\Gamma$ on a standard probability space $(X,\mu)$ is called \emph{essentially free} if a.e.\ $x \in X$ has a trivial stabilizer and is called \emph{ergodic} if the only $\Gamma$-invariant measurable subsets of $X$ have measure $0$ or $1$. Two free ergodic probability measure preserving (p.m.p.) actions $\Gamma \actson (X,\mu)$ and $\Lambda \actson (Y,\eta)$ are called
\begin{itemize}
\item \emph{conjugate,} if there exists an isomorphism of groups $\delta : \Gamma \recht \Lambda$ and an isomorphism of probability spaces $\Delta : X \recht Y$ such that $\Delta(g \cdot x) = \delta(g) \cdot \Delta(x)$ for all $g \in \Gamma$ and a.e.\ $x \in X$;
\item \emph{orbit equivalent,} if there exists an isomorphism of probability spaces $\Delta : X \recht Y$ such that $\Delta( \Gamma \cdot x) = \Lambda \cdot \Delta(x)$ for a.e.\ $x \in X$;
\item \emph{stably orbit equivalent,} if there exists a nonsingular isomorphism $\Delta : \cU \recht \cV$ between non-negligible measurable subsets $\cU \subset X$ and $\cV \subset Y$ such that $\Delta(\Gamma \cdot x \cap \cU) = \Lambda \cdot \Delta(x) \cap \cV$ for a.e.\ $x \in \cU$. Such a $\Delta$ automatically scales the measure by the constant $\eta(\cV)/\mu(\cU)$, called the \emph{compression constant} of the stable orbit equivalence.
\end{itemize}

We say that two p.m.p.\ actions $\Gamma \actson (X_i,\mu_i)$ of the same group are \emph{isomorphic} if they are conjugate w.r.t.\ the identity isomorphism $\id : \Gamma \recht \Gamma$, i.e.\ if there exists an isomorphism of probability spaces $\Delta : X_0 \recht X_1$ such that $\Delta(g \cdot x) = g \cdot \Delta(x)$ for all $g \in \Gamma$ and a.e.\ $x \in X_0$.

Recall that for every countable group $\Gamma$ and standard probability space $(X_0,\mu_0)$, the Bernoulli action of $\Gamma$ with base space $(X_0,\mu_0)$ is the action $\Gamma \actson X_0^\Gamma$ on the infinite product $X_0^\Gamma$ equipped with the product probability measure, given by $(g \cdot x)_h = x_{hg}$ for all $g,h \in \Gamma$ and $x \in X_0^\Gamma$. If $\Gamma$ is an infinite group and $(X_0,\mu_0)$ is not reduced to a single atom of mass $1$, then $\Gamma \actson X_0^\Gamma$ is essentially free and ergodic.

\subsection*{Statement of the main results}

We first give an elementary proof for the following theorem of Lewis Bowen.

\begin{thmstar}[Bowen \cite{Bo09a,Bo09b}] \label{thm.A}
For fixed $n$ and varying base probability space $(X_0,\mu_0)$ the Bernoulli actions $\F_n \actson X_0^{\F_n}$ are orbit equivalent.

If also $n$ varies, the Bernoulli actions $\F_n \actson X_0^{\F_n}$ and $\F_m \actson Y_0^{\F_m}$ are stably orbit equivalent with compression constant $(n-1)/(m-1)$.
\end{thmstar}

Next we study factors of Bernoulli actions and prove the following result.

\begin{thmstar}\label{thm.B}
If $\Gamma = \Lambda_1 * \cdots * \Lambda_n$ is the free product of $n$ infinite amenable groups and if $K$ is a nontrivial second countable compact group equipped with its normalized Haar measure, then the factor action $\Gamma \actson K^{\Gamma}/K$ of the Bernoulli action $\Gamma \actson K^\Gamma$ by the diagonal translation action of $K$ is isomorphic with a Bernoulli action of $\Gamma$. In particular, keeping $n$ fixed and varying the $\Lambda_i$ and $K$, all the actions $\Gamma \actson K^\Gamma/K$ are orbit equivalent.

In the particular case where $\Gamma = \F_n$, the action $\Gamma \actson K^{\Gamma}/K$ is isomorphic with the Bernoulli action $\Gamma \actson (K \times \cdots \times K)^\Gamma$ whose base space is an $n$-fold direct product of copies of $K$.
\end{thmstar}

\subsection*{Acknowledgment}

We are extremely grateful to Lewis Bowen for his remarks on the first versions of this article. Initially we only proved that the factors of Bernoulli actions of $\F_n$ in Theorem \ref{thm.B} are orbit equivalent with a Bernoulli action of $\F_n$. Lewis Bowen remarked that it was unknown whether these actions are actually isomorphic to Bernoulli actions. Triggered by this remark we proved the above version of Theorem \ref{thm.B}.

\section{Preliminaries}\label{sec.prelim}

Let $(X,\mu)$ and $(Y,\eta)$ be standard probability spaces. We call $\Delta$ a \emph{probability space isomorphism} between $(X,\mu)$ and $(Y,\eta)$ if $\Delta$ is a measure preserving Borel bijection between conegligible subsets of $X$ and $Y$. We call $\Delta$ a \emph{nonsingular isomorphism} if $\Delta$ is a null set preserving Borel bijection between conegligible subsets of $X$ and $Y$.

Given a sequence of standard probability spaces $(X_n,\mu_n)$, we consider the \emph{infinite product} $X = \prod_n X_n$ equipped with the infinite product measure $\mu$. Then, $(X,\mu)$ is a standard probability space. The coordinate maps $\pi_n : X \recht X_n$ are measure preserving and independent. Moreover, the Borel $\sigma$-algebra on $X$ is the smallest $\sigma$-algebra such that all $\pi_n$ are measurable.

Conversely, assume that $(Y,\eta)$ is a standard probability space and that $\theta_n : Y \recht X_n$ is a sequence of Borel maps. Then, the following two statements are equivalent.
\begin{enumerate}
\item There exists an isomorphism of probability spaces $\Delta : Y \recht X$ such that $\pi_n(\Delta(y)) = \theta_n(y)$ for a.e.\ $y \in Y$.
\item The maps $\theta_n$ are measure preserving and independent, and the $\sigma$-algebra on $Y$ generated by the maps $\theta_n$ equals the entire Borel $\sigma$-algebra of $Y$ up to null sets.
\end{enumerate}
The proof of this equivalence is standard: if the $\theta_n$ satisfy the conditions in 2, one defines $\Delta(y)_n := \theta_n(y)$.

Assume that $\Gamma \actson (X,\mu)$ and $\Lambda \actson (Y,\eta)$ are essentially free ergodic p.m.p.\ actions. Assume that $\Delta : X \recht Y$ is an orbit equivalence. By essential freeness, we obtain the a.e.\ well defined Borel map $\om : \Gamma \times X \recht \Lambda$ determined by
$$\Delta(g \cdot x) = \om(g,x) \cdot \Delta(x) \quad\text{for all}\;\; g \in \Gamma \;\;\text{and a.e.}\;\; x \in X \; .$$
Then, $\om$ is a \emph{$1$-cocycle} for the action $\Gamma \actson (X,\mu)$ with values in the group $\Lambda$. In general, whenever $\cG$ is a Polish group and $\Gamma \actson (X,\mu)$ is a p.m.p.\ action, we call a Borel map $\om : \Gamma \times X \recht \cG$ a $1$-cocycle if $\om$ satisfies
$$\om(gh,x) = \om(g,h \cdot x) \, \om(h,x) \quad\text{for all}\;\; g,h \in \Gamma \;\;\text{and a.e.}\;\; x \in X \; .$$
Two $1$-cocycles $\om,\om' : \Gamma \times X \recht \cG$ are called \emph{cohomologous} if there exists a Borel map $\vphi : X \recht \cG$ such that
$$\om'(g,x) = \vphi(g \cdot x) \, \om(g,x) \, \vphi(x)^{-1} \quad\text{for all}\;\; g \in \Gamma \;\;\text{and a.e.}\;\; x \in X \; .$$

Also a \emph{stable orbit equivalence} gives rise to a $1$-cocycle, as follows. So assume that $\Gamma \actson (X,\mu)$ and $\Lambda \actson (Y,\eta)$ are essentially free ergodic p.m.p.\ actions and that $\Delta : \cU \recht \cV$ is a nonsingular isomorphism between the nonnegligible subsets $\cU \subset X$ and $\cV \subset Y$, such that $\Delta(\cU \cap \Gamma \cdot x) = \cV \cap \Lambda \cdot \Delta(x)$ for a.e.\ $x \in \cU$. To define the Zimmer $1$-cocycle $\om : \Gamma \times X \recht \Lambda$, one first uses the ergodicity of $\Gamma \actson (X,\mu)$ to choose a Borel map $p : X \recht \cU$ satisfying $p(x) \in \Gamma \cdot x$ for a.e.\ $x \in X$. Then, $\om : \Gamma \times X \recht \Lambda$ is uniquely defined such that
$$\Delta(p(g \cdot x)) = \om(g,x) \cdot \Delta(p(x)) \quad\text{for all}\;\; g \in \Gamma \;\;\text{and a.e.}\;\; x \in X \; .$$
One checks easily that $\om$ is a $1$-cocycle and that, up to cohomology, $\om$ does not depend on the choice of $p : X \recht \cU$.

In this article, we often use $1$-cocycles for p.m.p.\ actions $\Gamma \actson (X,\mu)$ of a free product group $\Gamma = \Gamma_1 * \Gamma_2$. Given $1$-cocycles $\om_i : \Gamma_i \times X \recht \cG$, one checks easily that there is a unique $1$-cocycle $\om : \Gamma \times X \recht \cG$, up to equality a.e., satisfying $\om(g,x) = \om_i(g,x)$ for all $g \in \Gamma_i$ and a.e.\ $x \in X$.

\section{Orbit equivalence of co-induced actions}

Let $\Lambda \actson (X,\mu)$ be a p.m.p.\ action. Assume that $\Lambda < G$ is a subgroup. The co-induced action of $\Lambda \actson X$ to $G$ is defined as follows. Choose a map $r : G \recht \Lambda$ such that $r(\lambda g) = \lambda r(g)$ for all $g \in G,\lambda \in \Lambda$ and such that $r(e) = e$. Note that the choice of such a map $r$ is equivalent to the choice of a section $\theta : \Lambda \backslash \Gamma \recht \Gamma$ satisfying $\theta(\Lambda e) = e$. Indeed, the formula $g = r(g) \, \theta(\Lambda g)$ provides the correspondence between $\theta$ and $r$.

Once we have chosen $r : G \recht \Lambda$, we can define a $1$-cocycle $\Om : \Lambda \backslash G \times G \recht \Lambda$ for the right action of $G$ on $\Lambda \backslash G$, given by $\Omega(\Lambda k,g) = r(k)^{-1} r(kg)$ for all $g,k \in G$.

Classically, whenever $\om : G \times X \recht \Lambda$ is a $1$-cocycle for an action of $G$ on $X$, we can induce an action $\Lambda \actson Y$ to an action $G \actson X \times Y$ given by $g \cdot (x,y) = (g \cdot x, \om(g,x) \cdot y)$.

The co-induced action is defined by a similar formula. So assume that $\Lambda \actson (X,\mu)$ is a p.m.p.\ action and that $\Lambda < G$ is a subgroup. Choose $r : G \recht \Lambda$ with the associated $1$-cocycle $\Om : \Lambda \backslash G \times G \recht \Lambda$, as above. Then the formula
$$G \actson X^{\Lambda \backslash G} \quad\text{where}\quad (g \cdot y)_{\Lambda k} = \Omega(\Lambda k,g) \cdot y_{\Lambda kg}$$
yields a well defined action of $G$ on the product probability space $X^{\Lambda \backslash G}$. It is easy to check that $G \actson X^{\Lambda \backslash G}$ is a p.m.p.\ action and that $(\lambda \cdot y)_{\Lambda e} = \lambda \cdot y_{\Lambda e}$ for all $\lambda \in \Lambda$ and $y \in X^{\Lambda \backslash G}$. A different choice of $r : G \recht \Lambda$ leads to a cohomologous $1$-cocycle $\Om$ and hence an isomorphic action.

Given a subgroup $\Lambda < G$, a subset $I \subset G$ is called a \emph{right transversal} of $\Lambda < G$ if $I \cap \Lambda g$ is a singleton for every $g \in G$.

Up to isomorphism the co-induced action can be characterized as the unique p.m.p.\ action $G \actson Y$ for which there exists a measure preserving map $\rho : Y \recht X$ with the following properties.
\begin{enumerate}
\item $\rho(\lambda \cdot y) = \lambda \cdot \rho(y)$ for all $\lambda \in \Lambda$ and a.e.\ $y \in Y$.
\item The factor maps $y \mapsto \rho(g \cdot y)$, $g \in G$, generate the Borel $\sigma$-algebra on $Y$, up to null sets.
\item If $I \subset G$ is a right transversal of $\Lambda < G$, then the maps $y \mapsto \rho(g \cdot y)$, $g \in I$, are independent.
\end{enumerate}
To prove this characterization, first observe that the co-induced action satisfies properties 1, 2 and 3 in a canonical way, with $\rho(y) = y_{\Lambda e}$. Conversely assume that $G \actson Y$ satisfies these properties. Fix a right transversal $I \subset G$ for $\Lambda < G$, with $e \in I$. Combining properties 1 and 2, we see that the factor maps $y \mapsto \rho(g \cdot y)$, $g \in I$, generate the Borel $\sigma$-algebra on $Y$, up to null sets. A combination of property 3 and the characterization of product probability spaces in Section \ref{sec.prelim} then provides the isomorphism of probability spaces $\Delta : Y \recht X^{\Lambda \backslash G}$ given by $\Delta(y)_{\Lambda g} = \rho(g \cdot y)$ for all $y \in Y$, $g \in I$. The right transversal $I \subset G$ for $\Lambda < G$ allows to uniquely define the map $r : G \recht \Lambda$ such that $r(\lambda g) = \lambda$ for all $\lambda \in \Lambda$ and $g \in I$. This choice of $r$ provides a formula for the co-induced action $G \actson X^{\Lambda \backslash G}$. It is easy to check that $\Delta(g \cdot y) = g \cdot \Delta(y)$ for all $g \in G$ and a.e.\ $y \in Y$.

\begin{remark}\label{rem.bete}\mbox{}\\
{\bf 1.} The above characterization of the co-induced action yields the following result that we use throughout the article: the co-induction of the Bernoulli action $\Lambda \actson (X_0,\mu_0)^\Lambda$ is isomorphic with the Bernoulli action $G \actson (X_0,\mu_0)^G$. Indeed, the Bernoulli action $G \actson (X_0,\mu_0)^G$, together with the canonical factor map $X_0^G \recht X_0^\Lambda$, satisfies the above characterization of the co-induced action.

{\bf 2.} In certain cases, for instance if $G = \Gamma * \Lambda$, there exists a group homomorphism $\pi : G \recht \Lambda$ satisfying $\pi(\lambda) = \lambda$ for all $\lambda \in \Lambda$. Then $r : G \recht \Lambda$ can be taken equal to $\pi$ and the co-induced action $G \actson X^{\Lambda \backslash G}$ is of the form $(g \cdot y)_{\Lambda k} = \pi(g) \cdot y_{\Lambda k g}$ for all $g,k \in G$ and $y \in X^{\Lambda \backslash G}$.

{\bf 3.} We often make use of diagonal actions: if $\Lambda \actson (X,\mu)$ and $\Lambda \actson (Y,\eta)$ are p.m.p.\ actions, we consider the diagonal action $\Lambda \actson X \times Y$ given by $\lambda \cdot (x,y) = (\lambda \cdot x,\lambda \cdot y)$. We make the following simple observation: if $\Lambda < G$ and if we denote by $G \actson \Xtil$, resp.\ $G \actson \Ytil$, the co-induced actions of $\Lambda \actson X$, resp.\ $\Lambda \actson Y$, to $G$, then the co-induced action of the diagonal action $\Lambda \actson X \times Y$ to $G$ is precisely the diagonal action $G \actson \Xtil \times \Ytil$.

{\bf 4.} Assume that $\Lambda \actson (X,\mu)$ is a p.m.p.\ action and that $\Lambda < G$ is a subgroup. Denote by $G \actson Y$ the co-induced action and by $\rho : Y \recht X$ the canonical $\Lambda$-equivariant factor map. Whenever $\Delta_0 : X \recht X$ is a p.m.p.\ automorphism that commutes with the $\Lambda$-action, there is a unique p.m.p.\ automorphism $\Delta : Y \recht Y$, up to equality a.e., that commutes with the $G$-action and such that $\rho(\Delta(y)) = \Delta_0(\rho(y))$ for a.e.\ $y \in Y$. Writing $Y = X^{\Lambda \backslash \Gamma}$, the automorphism $\Delta$ is just the diagonal product of copies of $\Delta_0$. Later we use this easy observation to canonically lift a p.m.p.\ action $K \actson (X,\mu)$ of a compact group $K$, commuting with the $\Lambda$-action, to a p.m.p.\ action $K \actson Y$ that commutes with the $G$-action. Moreover, $\rho$ becomes $(\Lambda \times K)$-equivariant. Writing $Y = X^{\Lambda \backslash \Gamma}$, the action $K \actson Y$ is the diagonal $K$-action.
\end{remark}

We prove that orbit equivalence is preserved under co-induction to a free product. We actually show that the preservation is ``$K$-equivariant'' in a precise way that will be needed in the proof of Theorem \ref{thm.B}. The case where $K = \{e\}$, i.e.\ co-induction from $\Lambda$ to $\Gamma * \Lambda$, is due to Lewis Bowen \cite{Bo09a}. Recall that similarly as in the case of countable groups, a p.m.p.\ action $G \actson (X,\mu)$ of a second countable locally compact group $G$ is called essentially free if a.e.\ $x \in X$ has a trivial stabilizer (cf.\ Lemma \ref{lem.compact-free} in the appendix).

\begin{theorem}\label{thm.coinduced}
Let $\Lambda_0,\Lambda_1$ and $\Gamma$ be countable groups and $K$ a compact second countable group. Assume that $\Lambda_i \times K \actson (X_i,\mu_i)$ are essentially free p.m.p.\ actions. Denote $G_i := \Gamma * \Lambda_i$ and denote by $G_i \actson Y_i$ the co-induced action of $\Lambda_i \actson X_i$ to $G_i$, together with the natural actions $K \actson Y_i$ that commute with $G_i \actson Y_i$ (see Remark \ref{rem.bete}.4).
\begin{itemize}
\item If the actions $\Lambda_i \actson X_i/K$ are orbit equivalent, then the actions $G_i \actson Y_i/K$ are orbit equivalent.
\item If the actions $\Lambda_i \actson X_i/K$ are conjugate w.r.t.\ the group isomorphism $\delta : \Lambda_0 \recht \Lambda_1$, then the actions $G_i \actson Y_i/K$ are conjugate w.r.t.\ the group isomorphism $\id * \delta : G_0 \recht G_1$.
\end{itemize}
\end{theorem}
\begin{proof}
We start by proving the first item of the theorem.

Let $\Delta_0 : X_0/K \recht X_1/K$ be an orbit equivalence between the actions $\Lambda_i \actson X_i/K$. Denote by $x \mapsto \xb$ the factor map from $X_i$ to $X_i / K$. Since $K$ acts essentially freely on $X_i$ and $K$ is compact, Lemma \ref{lem.compact-free} in the appendix provides measurable maps $\theta_i : X_i \recht K$ satisfying $\theta_i(k \cdot x) = k \theta_i(x)$ a.e.\ and such that
$$\Theta_i : X_i \recht K \times X_i/K : x \mapsto (\theta_i(x) , \xb)$$
is a measure preserving isomorphism. Defining $\Delta := \Theta_1^{-1} \circ (\id \times \Delta_0) \circ \Theta_0$, we have found a measure preserving isomorphism $\Delta : X_0 \recht X_1$ that is $K$-equivariant and satisfies $\Delta((\Lambda_0 \times K) \cdot x) = (\Lambda_1 \times K) \cdot \Delta(x)$ for a.e.\ $x \in X_1$. Using this $\Delta$ we may assume that $\Lambda_0, \Lambda_1$ and $K$ act on the same probability space $(X,\mu)$ such that the $K$-action commutes with both the $\Lambda_i$-actions and such that $(\Lambda_0 \times K) * x = (\Lambda_1 \times K) \cdot x$ for a.e.\ $x \in X$. Here and in what follows, we denote the action of $\Lambda_0 \times K$ by $*$ and the action of $\Lambda_1 \times K$ by $\cdot$. We have $k * x = k \cdot x$ for all $k \in K$ and a.e.\ $x \in X$.

Write $Y = X^{\Lambda_1 \backslash \Gamma * \Lambda_1}$ and denote by $\cdot$ the co-induced action $G_1 \actson Y$ of $\Lambda_1 \actson X$ to $G_1$. Also denote by $\cdot$ the diagonal action $K \actson Y$, which commutes with $G_1 \actson Y$. Define the $(\Lambda_1 \times K)$-equivariant factor map $\rho : Y \recht X : \rho(y) = y_{\Lambda_1 e}$.

Define the Zimmer $1$-cocycles
\begin{align*}
& \eta : \Lambda_0 \times X \recht \Lambda_1 \times K : \eta(\lambda_0,x) \cdot x = \lambda_0 * x \quad\text{for a.e.}\;\; x \in X_1, \lambda_0 \in \Lambda_0 \; ,\\
& \eta' : \Lambda_1 \times X \recht \Lambda_0 \times K : \eta'(\lambda_1,x) * x = \lambda_1 \cdot x \quad\text{for a.e.}\;\; x \in X_1, \lambda_1 \in \Lambda_1 \; .
\end{align*}
Since the $\Lambda_0$-action commutes with the $K$-action on $X$, we have that
\begin{equation}\label{eq.my-formula}
\eta(\lambda_0, k * x) = k \eta(\lambda_0,x) k^{-1} \quad\text{for all}\;\; k \in K, \lambda_0 \in \Lambda_0 \;\;\text{and a.e.}\;\; x \in X \; .
\end{equation}
We define a new action $G_0 \actson Y$ denoted by $*$ and determined by
$$\gamma * y  = \gamma \cdot y \;\;\text{for}\;\; \gamma \in \Gamma, y \in Y \quad\text{and}\quad \lambda_0 * y = \eta(\lambda_0, \rho(y)) \cdot y \;\;\text{for}\;\; \lambda_0 \in \Lambda_0, y \in Y \; .$$
Because of \eqref{eq.my-formula}, the action $G_0 \actson Y$ commutes with $K \actson Y$.

Define $\om : G_0 \times Y \recht G_1 \times K$ as the unique $1$-cocycle for the action $G_0 \overset{*}{\actson} Y$ satisfying $\om(\gamma,y) = \gamma$ for all $\gamma \in \Gamma$ and $\om(\lambda_0,y) = \eta(\lambda_0,\rho(y))$ for all $\lambda_0 \in \Lambda_0$. Then the equality $g * y = \om(g,y) \cdot y$ holds when $g \in \Gamma$ and when $g \in \Lambda_0$. So the same equality holds for all $g \in G_0$ and a.e.\ $y \in Y$. In particular $G_0 * \yb \subset G_1 \cdot \yb$ for a.e.\ $\yb \in Y/K$.

Define $\om' : G_1 \times Y \recht G_0 \times K$ as the unique $1$-cocycle satisfying $\om'(\gamma,y) = \gamma$ for all $\gamma \in \Gamma$ and $\om'(\lambda_1,y) = \eta'(\lambda_1,\rho(y))$ for all $\lambda_1 \in \Lambda_1$. As above, it follows that $g \cdot y = \om'(g,y) * y$ for all $g \in G_1$ and a.e.\ $y \in Y$. Hence, $G_1 \cdot \yb \subset G_0 * \yb$ for a.e.\ $\yb \in Y/K$. We already proved the converse inclusion so that $G_1 \cdot \yb = G_0 * \yb$ for a.e.\ $y \in Y/K$.

We prove now that the action $G_0 \overset{*}{\actson} Y$ together with the $\Lambda_0$-equivariant factor map $\rho : Y \recht X$
satisfies the abstract characterization for the co-induced action of $\Lambda_0 \actson X$ to $G_0$. Once this is proven, the theorem follows because $\rho$ is moreover $K$-equivariant and the action $G_0 \actson Y$ commutes with the $K \actson Y$ (see Remark \ref{rem.bete}.4).

We first need to prove that the maps $y \mapsto \rho(g*y)$ are independent and identically distributed when $g$ runs through a right transversal of $\Lambda_0 \subset G_0$. If $g \in G_i = \Gamma * \Lambda_i$, denote by $|g|$ the number of letters from $\Gamma - \{e\}$ that appear in a reduced expression of $g$. By convention, put $|g| = 0$ if $g \in \Lambda_i$. Define the subsets $I_n \subset G_0$ given by $I_0 := \{e\}$ and
$$I_n := \bigl\{g \in G_0 \;\big| \; |g|=n \;\;\text{and the leftmost letter of a reduced expression of $g$ belongs to $\Gamma-\{e\}$}\;\bigr\} \; .$$
Similarly define $J_n \subset G_1$ and note that $\bigcup_{n=0}^\infty J_n$ is a right transversal for $\Lambda_1 < \Gamma * \Lambda_1$.
So, in the construction of the co-induced action, we can choose the $\Lambda_1$-equivariant map $r : G_1 \recht \Lambda_1$ such that $r(g) = e$ for all $g \in J_n$ and all $n \in \N$. Hence $(g \cdot y)_{\Lambda_1 e} = y_{\Lambda_1 g}$ for all $g \in J_n$, $n \in \N$ and a.e.\ $y \in Y$. For $j \in \Lambda_1 \backslash G_1$ we put $|j|=n$ if $j = \Lambda_1 g$ with $g \in J_n$.

Denote $\om(g,y) = (\om_1(g,y),\om_K(g,y))$ with $\om_1(g,y) \in G_1$ and $\om_K(g,y) \in K$. Similarly write $\eta(\lambda,x) = (\eta_1(\lambda,x),\eta_K(\lambda,x))$. Note that for $\lambda \in \Lambda_0 - \{e\}$ we have $\eta_1(\lambda,x) \neq e$ for a.e.\ $x \in X$. Indeed, if $\eta_1(\lambda,x) = e$ for a fixed $\lambda \in \Lambda_0 - \{e\}$, then the element $(\lambda,\eta_K(\lambda,x)^{-1})$ of $\Lambda_0 \times K$ stabilizes $x$ and the essential freeness of $\Lambda_0 \times K \actson X$ implies that this can only happen for $x$ belonging to a negligible subset of $X$.
One then proves easily by induction on $n$ that
\begin{itemize}
\item for a.e.\ $y \in Y$ and all $n \in \N$, the map $g \mapsto \om_1(g,y)$ is a bijection of $I_n$ onto $J_n$,
\item for all $n \in \N, g \in I_n$, the map $y \mapsto \om(g,y)$ only depends on the coordinates $y_j$, $|j| \leq n-1$.
\end{itemize}
Since for all $g \in I_n$ we have $\om_1(g,y) \in J_n$, it follows that
\begin{equation}\label{eq.nuttig}
\rho(g*y) = (g*y)_{\Lambda_1 e} = (\om(g,y) \cdot y)_{\Lambda_1 e} = \om_K(g,y) \cdot y_{\Lambda_1 \om_1(g,y)}
\end{equation}
for all $n \in \N$, $g \in I_n$ and a.e.\ $y \in Y$. We now use Lemma \ref{lem.indep} to prove that for all $n \in \N$, the set
$\{y \mapsto \rho(g*y) \mid g \in I_n\}$ forms a family of independent random variables that are independent of the coordinates $y_j$, $|j| \leq n-1$, and that only depend on the coordinates $y_j$, $|j| \leq n$. More concretely, we write $\cJ_n = \{ \Lambda_1 g \mid |g| \leq n\}$ and we apply Lemma \ref{lem.indep} to the countable set $\cJ_n - \cJ_{n-1}$, the direct product
$$Z := X^{\cJ_{n-1}} \times X^{\cJ_n - \cJ_{n-1}}$$
and the family of measurable maps $\om_g : Z \recht K \times (\cJ_n - \cJ_{n-1})$ indexed by $g \in I_n$, only depending on the coordinates $y_j$, $j \in \cJ_{n-1}$ and given by
$$\om_g : y \mapsto (\om_K(g,y), \Lambda_1 \om_1(g,y)) \; .$$
Since $g \mapsto \om_1(g,y)$ is a bijection of $I_n$ onto $J_n$, we have that $g \mapsto \Lambda_1 \om_1(g,y)$ is a bijection of $I_n$ onto $\cJ_{n} - \cJ_{n-1}$. A
combination of Lemma \ref{lem.indep} and formula \eqref{eq.nuttig} then implies that $\{y \mapsto \rho(g*y) \mid g \in I_n\}$ is a family of independent random variables that are independent of the coordinates $y_j$, $j \in \cJ_{n-1}$. By construction, these random variables only depend on the coordinates $y_j$, $|j| \leq n$.
Having proven these statements for all $n \in \N$, it follows that $\{y \mapsto \rho(g*y) \mid g \in \bigcup_n I_n\}$ is a family of independent random variables.

Denote by $\cB_0$ the smallest $\sigma$-algebra on $Y$ such that $Y \recht X_1 : y \mapsto \rho(g * y)$ is $\cB_0$-measurable for all $g \in G_0$. It remains to prove that $\cB_0$ is the entire $\sigma$-algebra of $Y$. Note that by construction, the map $Y \recht Y : y \mapsto g * y$ is $\cB_0$-measurable for all $g \in G_0$. 
Since $\rho$ is $K$-equivariant and the actions $K \actson Y$ and $G_0 \actson Y$ commute, we also get that the map $y \mapsto k * y$ is $\cB_0$-measurable for every $k \in K$.  We must prove that $y \mapsto y_i$ is $\cB_0$-measurable for every $n \in \N$ and $i \in \Lambda_1 \backslash G_1$ with $|i|=n$. This follows by induction on $n$, because for all $g \in J_n$ we have
$$y_{\Lambda_1 g} = \rho(g \cdot y) = \rho(\om'(g,y) * y)$$
and because $y \mapsto \om'(g,y)$ only depends on the coordinates $y_j$, $|j| \leq n-1$.

To prove the second item of the theorem, it suffices to make the following observation. If the actions $\Lambda_i \actson X_i /K$ are conjugate w.r.t.\ the isomorphism $\delta : \Lambda_0 \recht \Lambda_1$, then in the proof of the first item, the Zimmer $1$-cocycle $\eta$ is of the form $\eta(\lambda_0,x) = (\delta(\lambda_0),\eta_K(\lambda_0,x))$. So the $1$-cocycle $\om : G_0 \times Y \recht G_1 \times K$ is of the form $\om(g,y) = ((\id * \delta)(g),\om_K(g,y))$. This immediately implies that the actions $G_i \actson Y_i/K$ are conjugate w.r.t.\ the isomorphism $\id * \delta$.
\end{proof}

\begin{corollary}[Bowen \cite{Bo09a}] \label{cor.indep-base}
For fixed $n$ and varying base probability space $(X_0,\mu_0)$ the Bernoulli actions $\F_n \actson X_0^{\F_n}$ are orbit equivalent.
\end{corollary}
\begin{proof}
By Remark \ref{rem.bete}.1, the co-induction of a Bernoulli action is again a Bernoulli action over the same base space. Let $X_0$ and $X_1$ be nontrivial base probability spaces. By Dye's theorem \cite{Dy58}, the Bernoulli actions $\Z \actson X_0^{\Z}$ and $\Z \actson X_1^\Z$ are orbit equivalent. By Theorem \ref{thm.coinduced} their co-induced actions to $\F_n = \F_{n-1} * \Z$ are orbit equivalent. But these co-induced actions are isomorphic to the Bernoulli actions $\F_n \actson X_i^{\F_n}$.
\end{proof}

We used the following easy independence lemma.

\begin{lemma}\label{lem.indep}
Let $(X,\mu)$ and $(X_0,\mu_0)$ be standard probability spaces and let $H \actson (X_0,\mu_0)$ be a measure preserving action.  Let $I$ be a countable set. Consider $Z = X \times X_0^I$ with the product probability measure. Assume that $\cF$ is a family of measurable maps $\om : Z \recht H \times I$. Write $\om(x,y) = (\om_1(x,y),\om_2(x,y))$. Assume that
\begin{itemize}
\item for almost every $z \in Z$, the map $\cF \recht I : \om \mapsto \om_2(z)$ is injective,
\item for every $\om \in \cF$, the map $z \mapsto \om(z)$ only depends on the variable $Z \recht X : (x,y) \mapsto x$.
\end{itemize}
Then, $\{(x,y) \mapsto \om_1(x,y) \cdot y_{\om_2(x,y)} \mid \om \in \cF\}$ is a family of independent identically $(X_0,\mu_0)$-distributed random variables that are independent of $(x,y) \mapsto x$.
\end{lemma}
\begin{proof}
Since the maps $\om \in \cF$ only depend on the variable $(x,y) \mapsto x$, we view $\om \in \cF$ as a map from $X$ to $H \times I$. We have to prove that $\{(x,y) \mapsto \om_1(x) \cdot y_{\om_2(x)} \mid \om \in \cF\}$ is a family of independent identically $(X_0,\mu_0)$-distributed random variables that are independent of $(x,y) \mapsto x$. But conditioning on $x \in X$, we get that the variables
$$X_0^I \recht X_0 : y \mapsto \om_1(x) \cdot y_{\om_2(x)}$$
are independent and $(X_0,\mu_0)$-distributed because the coordinates $\om_2(x)$, for $\om \in \cF$, are distinct elements of $I$ and because the action $H \actson X_0$ is measure preserving. So the lemma is proven.
\end{proof}

\section{Stable orbit equivalence of Bernoulli actions}

Denote by $a,b$ the standard generators of $\F_2$. Denote by $\langle a \rangle$ and $\langle b \rangle$ the subgroups of $\F_2$ generated by $a$, resp.\ $b$. Let $(X_0,\mu_0)$ be a standard probability space and consider the Bernoulli action $\F_2 \actson X_0^{\F_2}$ given by $(g \cdot x)_h = x_{hg}$.

Whenever $(X_0,\mu_0)$ is a probability space, the Bernoulli action $\Gamma \actson X_0^\Gamma$ can be characterized up to isomorphism as the unique p.m.p.\ action $\Gamma \actson X$ for which there exists a factor map $\pi : X \recht X_0$ such that the maps $x \mapsto \pi(g \cdot x)$, $g \in \Gamma$, are independent and generate, up to null sets, the whole $\sigma$-algebra of $X$.

We prove the stable orbit equivalence of Bernoulli actions as a combination of the following three lemmas. Fix $\kappa \in \N$, $\kappa \geq 2$, and denote $X_0 = \{0,\ldots,\kappa-1\}$ equipped with the uniform probability measure. Let $(Y_0,\eta_0)$ be any standard probability space (that is not reduced to a single atom). Denote by $r : \F_2 \recht \Z/\kappa \Z$ the group morphism determined by $r(a) = 0$ and $r(b) = 1$. Identify $X_0$ with $\Z/\kappa \Z$ and denote by $\cdot$ the action of $\Z/\kappa \Z$ on $X_0$ given by addition in $\Z/\kappa \Z$.

\begin{lemma}\label{lem.een}
Consider the action $\F_2 \actson X := X_0^{\langle b \rangle \backslash \F_2}$ given by $(g \cdot x)_{\langle b \rangle h} = r(g) \cdot x_{\langle b \rangle h g}$. Let $\F_2 \actson Y_0^{\F_2}$ be the Bernoulli action. Then the diagonal action $\F_2 \actson X \times Y_0^{\F_2}$ given by $g \cdot (x,y) = (g \cdot x,g \cdot y)$ is orbit equivalent with a Bernoulli action of $\F_2$.
\end{lemma}

\begin{lemma}\label{lem.twee}
The action $\F_2 \actson X$ defined in Lemma \ref{lem.een} is stably orbit equivalent with compression constant $1/\kappa$ with a Bernoulli action of $\F_{1 + \kappa}$.
\end{lemma}

\begin{lemma}\label{lem.drie}
Let $\Gamma \actson (X,\mu)$ be any free ergodic p.m.p.\ action of an infinite group $\Gamma$. Assume that $\kappa \in \N$ and that $\Gamma \actson X$ is stably orbit equivalent with compression constant $1/\kappa$ with a Bernoulli action of some countable group $\Lambda$. Let $(Y_0,\eta_0)$ be any standard probability space and $\Gamma \actson Y_0^\Gamma$ the Bernoulli action. Then also the diagonal action $\Gamma \actson X \times Y_0^\Gamma$ is stably orbit equivalent with compression constant $1/\kappa$ with a Bernoulli action of $\Lambda$.
\end{lemma}

\subsubsection*{Proof of Theorem \ref{thm.A}}

We already deduce Theorem \ref{thm.A} from the above three lemmas.

\begin{proof}[Proof of Theorem \ref{thm.A}]
We first prove that Lemmas \ref{lem.een}, \ref{lem.twee}, \ref{lem.drie} yield a Bernoulli action of $\F_2$ that is stably orbit equivalent with compression constant $1/\kappa$ with a Bernoulli action of $\F_{1+\kappa}$. Indeed, by Lemma \ref{lem.een} a Bernoulli action of $\F_2$ is orbit equivalent with the diagonal action $\F_2 \actson X \times Y_0^{\F_2}$. By Lemma \ref{lem.twee}, the action $\F_2 \actson X$ is stably orbit equivalent with compression constant $1/\kappa$ with a Bernoulli action of $\F_{1+\kappa}$. But then, Lemma \ref{lem.drie} says that the same holds for the diagonal action $\F_2 \actson X \times Y_0^{\F_2}$.

Combined with Corollary \ref{cor.indep-base} it follows that \emph{all} Bernoulli actions of $\F_2$ are stably orbit equivalent with all Bernoulli actions of $\F_m$, $m \geq 2$, with compression constant $1/(m-1)$. By transitivity of stable orbit equivalence, all Bernoulli actions of $\F_n$ and $\F_m$ are stably orbit equivalent with compression constant $(n-1)/(m-1)$.
\end{proof}

\subsubsection*{Proof of Lemma \ref{lem.een}}

\begin{proof}[Proof of Lemma \ref{lem.een}]
View $\Z$ as the subgroup of $\F_2$ generated by $b$. Let $\Z \actson Y_0^\Z$ be the Bernoulli action. Consider the action $\Z \actson X_0 \times Y_0^\Z$ given by $g \cdot (x,y) = (r(g) \cdot x, g \cdot y)$. Note that $\Z \actson X_0 \times Y_0^\Z$ is a free ergodic p.m.p.\ action. Using Remark \ref{rem.bete} (statements 1, 2 and 3), one gets that the action $\F_2 \actson X \times Y_0^{\F_2}$ given in the formulation of Lemma \ref{lem.een} is precisely the co-induction of $\Z \actson X_0 \times Y_0^\Z$ to $\F_2$. By Dye's theorem \cite{Dy58}, the free ergodic p.m.p.\ action $\Z \actson X_0 \times Y_0^{\Z}$ is orbit equivalent with a Bernoulli action of $\Z$. By Remark \ref{rem.bete}.1, the co-induction of the latter is a Bernoulli action of $\F_2$. So by Theorem \ref{thm.coinduced}, the action $\F_2 \actson X \times Y_0^{\F_2}$ is orbit equivalent with a Bernoulli action of $\F_2$.
\end{proof}

\subsubsection*{Proof of Lemma \ref{lem.twee}}

\begin{proof}[Proof of Lemma \ref{lem.twee}]
We have $X = X_0^{\langle b \rangle \backslash \F_2}$ and the action $\F_2 \actson X$ is given by $(g \cdot x)_{\langle b \rangle h} = r(g) \cdot x_{\langle b \rangle hg}$. Write $Z = X_0^\Z$ and denote by $\rho : X \recht Z$ the factor map given by $\rho(x)_n = x_{\langle b \rangle a^n}$. Denote by $\cdot$ the Bernoulli action $\Z \actson Z$ and note that $\rho(a^n \cdot x) = n \cdot \rho(x)$ for all $x \in X$ and $n \in \Z$.

Define the subsets $V_i$, $i=0,\ldots,\kappa-1$, of $Z$ given by $V_i := \{z \in Z \mid z_0 = i\}$. Similarly define $W_i \subset X$ given by $W_i = \rho^{-1}(V_i)$. Note that $W_0$ has measure $1/\kappa$. To prove the lemma we define a p.m.p.\ action of $\F_{1+\kappa}$ on $W_0$ such that $\F_{1+\kappa} * x = \F_2 \cdot x \cap W_0$ for a.e.\ $x \in W_0$ and such that $\F_{1+\kappa} \actson W_0$ is a Bernoulli action.

By Dye's theorem \cite{Dy58}, there exists a Bernoulli action $\Z \overset{*}{\actson} V_0$ such that $\Z * z = \Z \cdot z \cap V_0$ for a.e.\ $z \in V_0$. Denote by $\eta : \Z \times V_0 \recht \Z$ the corresponding $1$-cocycle for the $*$-action determined by $n * z = \eta(n,z) \cdot z$ for $n \in \Z$ and a.e.\ $z \in V_0$.

Since the Bernoulli action $\Z \overset{\cdot}{\actson} Z$ is ergodic and since all the subsets $V_i \subset Z$ have the same measure, we can choose measure preserving isomorphisms $\al_i : V_0 \recht V_i$ satisfying $\al_i(z) \in \Z \cdot z$ for a.e.\ $z \in \Z$ and take $\al_0$ to be the identity isomorphism. Let $\vphi^0_i : V_0 \recht \Z$ and $\psi^0_i : V_i \recht \Z$ be the maps determined by $\al_i(z) = \vphi^0_i(z) \cdot z$ for a.e.\ $z \in V_0$ and $\al_i^{-1}(z) = \psi^0_i(z) \cdot z$ for a.e.\ $z \in V_i$. Define the corresponding measure preserving isomorphisms $\theta_i : W_0 \recht W_i$ given by $\theta_i(x) = \vphi_i(x) \cdot x$ and $\theta_i^{-1}(x) = \psi_i(x) \cdot x$ where $\vphi_i(x) = a^{\vphi^0_i(\rho(x))}$ and $\psi_i(x) = a^{\psi^0_i(\rho(x))}$.

Denote by $a$ and $b_i$, $i = 0,\ldots,\kappa-1$, the generators of $\F_{1+\kappa}$. Define the p.m.p.\ action $\F_{1+\kappa} \overset{*}{\actson} W_0$ given by
$$a^n * x = a^{\eta(n,\rho(x))} \cdot x \quad\text{and}\quad b_i * x = \theta_{i+1}^{-1} (b \cdot \theta_i(x)) \quad\text{for all}\;\; x \in W_0 \; .$$
Note that the action is well defined: if $x \in W_0$, then $\theta_i(x) \in W_i$ and hence $b \cdot \theta_i(x) \in W_{i+1}$. We use the convention that $W_\kappa = W_0$ and $\theta_\kappa = \id$. Observe that $\rho(a^n * x) = n * \rho(x)$ for all $n \in \Z$ and a.e.\ $x \in W_0$.

It remains to prove that $\F_{1+\kappa} * x = \F_2 \cdot x \cap W_0$ for a.e.\ $x \in W_0$ and that $\F_{1+\kappa} \actson W_0$ is a Bernoulli action.

Denote by $\om : \F_{1+\kappa} \times W_0 \recht \F_2$ the unique $1$-cocycle for the $*$-action determined by
$$\om(a^n,x) = a^{\eta(n,\rho(x))} \quad\text{and}\quad \om(b_i,x) = \psi_{i+1}(b \cdot \theta_i(x)) \, b \, \vphi_i(x) \; .$$
By construction, the formula $g * x = \om(g,x) \cdot x$ holds for all $g \in \{a,b_0,\ldots,b_{\kappa-1}\}$ and a.e.\ $x \in W_0$. Since $\om$ is a $1$-cocycle for the action $\F_{1+\kappa} \overset{*}{\actson} W_0$, the same formula holds for all $g \in \F_{1+\kappa}$ and a.e.\ $x \in W_0$. In particular, $\F_{1+\kappa} * x \subset \F_2 \cdot x \cap W_0$ for a.e.\ $x \in W_0$. To prove the converse inclusion we define the inverse $1$-cocycle for $\om$.

Define $q_0 : Z \recht V_0$ given by $q_0(z) = \al_i^{-1}(z)$ when $z \in V_i$. Denote by $\eta' : \Z \times Z \recht \Z$ the $1$-cocycle for the $\cdot$-action determined by $q_0(n \cdot z) = \eta'(n,z) * q_0(z)$. Whenever $z \in V_0$, we have $z = q_0(z)$ and hence
\begin{equation}\label{eq.sven}
\eta'(\eta(n,z),z) * z = \eta'(\eta(n,z),z) * q_0(z) = q_0(\eta(n,z) \cdot z) = q_0(n*z) = n * z \; .
\end{equation}
Since $*$ is an essentially free action of $\Z$, it follows that $\eta'(\eta(n,z),z) = n$ for all $n \in \Z$ and a.e.\ $z \in V_0$.

Denote by $\om' : \F_2 \times X \recht \F_{1+\kappa}$ the unique $1$-cocycle for the $\cdot$-action determined by
$$\om'(a^n,x) = a^{\eta'(n,\rho(x))} \quad\text{for $n \in \Z$ and a.e.\ $x \in X$, and}\quad \om'(b,x) = b_i \quad\text{for a.e.\ $x \in W_i$.}$$
Define $q : X \recht W_0$ given by $q(x) = \theta_i^{-1}(x)$ when $x \in W_i$. Note that $\rho(q(x)) = q_0(\rho(x))$ for a.e.\ $x \in X$.
We prove that $q(g \cdot x) = \om'(g,x) * q(x)$ for all $g \in \F_2$ and a.e.\ $x \in X$. If $g = a^n$ for some $n \in \Z$, we know that both $q(g \cdot x)$ and $\om'(g,x) * q(x)$ belong to $\langle a \rangle \cdot x$. So to prove that they are equal, it suffices to check that they have the same image under $\rho$. The following computation shows that this is indeed the case.
\begin{align*}
& \rho(q(a^n \cdot x)) = q_0(\rho(a^n \cdot x)) = q_0(n \cdot \rho(x)) = \eta'(n,\rho(x)) * q_0(\rho(x)) \quad\text{while,} \\
& \rho(\om'(a^n,x) * q(x)) = \rho(a^{\eta'(n,\rho(x))} * q(x)) = \eta'(n,\rho(x)) * \rho(q(x)) = \eta'(n,\rho(x)) * q_0(\rho(x)) \; .
\end{align*}
Since by definition of the action $*$ we have that $b_i * \theta_i^{-1}(x) = \theta_{i+1}^{-1}(b \cdot x)$ whenever $x \in W_i$, the formula $\om'(g,x) * q(x) = q(g \cdot x)$ also holds when $g = b$. Hence, the same formula holds for all $g \in \F_2$ and a.e.\ $x \in X$. In particular, $\F_2 \cdot x \cap W_0 \subset \F_{1+\kappa} * x$ for a.e.\ $x \in W_0$. The converse inclusion was already proven above. Hence, $\F_{1+\kappa} * x = \F_2 \cdot x \cap W_0$ for a.e.\ $x \in W_0$.

Denote by $\cJ \subset \F_{1+\kappa}$ the union of $\{e\}$ and all the reduced words that start with one of the letters $b_i^{\pm 1}$, $i = 0,\ldots,\kappa-1$. Note that $\cJ$ is a right transversal for $\langle a \rangle < \F_{1+\kappa}$. It remains to prove that
$$\{ W_0 \recht V_0 : x \mapsto \rho(g * x) \mid g \in \cJ\}$$
is a family of independent random variables that generate, up to null sets, the whole $\si$-algebra on $W_0$. Indeed, we already know that $\Z \overset{*}{\actson} V_0$ is a Bernoulli action so that it will follow that $\F_{1+\kappa} \actson W_0$ is the co-induction of a Bernoulli action, hence a Bernoulli action itself (see Remark \ref{rem.bete}.1).

We equip both $\F_2$ and $\F_{1+\kappa}$ with a length function. For $g \in \F_2$ we denote by $|g|$ the number of letters $b^{\pm 1}$ appearing in the reduced expression of $g$, while for $g \in \F_{1+\kappa}$ we denote by $|g|$ the number of letters $b_i^{\pm 1}$, $i = 0,\ldots,\kappa-1$, appearing in the reduced expression of $g$. By induction on the length of $g$, one easily checks that $|\om(g,x)| \leq |g|$ for all $g \in \F_{1+\kappa}$ and a.e.\ $x \in W_0$, and that $|\om'(g,x)| \leq |g|$ for all $g \in \F_2$ and a.e.\ $x \in X$.

We next claim that
\begin{equation}\label{eq.inverse}
\om'(\om(g,x),x) = g \quad\text{for all $g \in \F_{1+\kappa}$ and a.e.\ $x \in W_0$.}
\end{equation}
Once this claim is proven, it follows that $|\om(g,x)| = |g|$ for all $g \in \F_{1+\kappa}$ and a.e.\ $x \in W_0$~: indeed, the strict inequality $|\om(g,x)| < |g|$ would lead to the contradiction
$$|g| = |\om'(\om(g,x),g)| \leq |\om(g,x)| < |g| \; .$$
First note that for $g = a^n$ formula \eqref{eq.inverse} follows immediately from \eqref{eq.sven}. So it remains to prove \eqref{eq.inverse} when $g = b_i$. First observe that $\om'(\vphi_i(x),x) = e$ for a.e.\ $x \in W_0$. Indeed,
$$\om'(\vphi_i(x),x) * x = q(\vphi_i(x) \cdot x) = q(\theta_i(x)) = x$$
and since the $*$-action of $\langle a \rangle$ on $W_0$ is essentially free, it follows that $\om'(\vphi_i(x),x) = e$. Similarly, $\om'(\psi_i(x),x) = e$ for a.e.\ $x \in W_i$. Take $x \in W_0$ and write $x' := b \vphi_i(x) \cdot x$. Note that $x' = b \cdot \theta_i(x)$ and that $x' \in W_{i+1}$. So,
$$\om'(\om(b_i,x),x) = \om'(\psi_{i+1}(x') \, b \, \vphi_i(x) , x) = \om'(\psi_{i+1}(x'), x') \, \om'(b, \theta_i(x)) \, \om'(\vphi_i(x),x) = e \, b_i \, e = b_i \; .$$
So \eqref{eq.inverse} holds for $g = a^n$ and $g = b_i$. Hence \eqref{eq.inverse} holds for all $g \in \F_{1+\kappa}$. Note that \eqref{eq.inverse} implies that the action $\F_{1+\kappa} \overset{*}{\actson} W_0$ is essentially free. Indeed, if $g \in \F_{1+\kappa}$, $x \in W_0$ and $g * x = x$, it follows that $\om(g,x) \cdot x = x$. Since the $\cdot$-action is essentially free, we conclude that $\om(g,x) = e$. But then by \eqref{eq.inverse}
$$g = \om'(\om(g,x),x) = \om'(e,x) = e \; .$$

Define the subsets $\cC(n) \subset \langle b \rangle \backslash \F_2$ given by $\cC(n) := \{\langle b \rangle g \mid g \in \F_2, |g| \leq n\}$. Also define $\cJ_n := \{g \in \cJ \mid |g| \leq n\}$. We prove by induction on $n$ that the following two statements hold.
\begin{itemize}
\item[$1_n.$] If $g \in \F_{1+\kappa}$ and $|g| \leq n$, then $x \mapsto \om(g,x)$ only depends on the coordinates $x_i$, $i \in \cC(n)$.
\item[$2_n.$] The set $\{W_0 \recht V_0 \mid x \mapsto \rho(g * x) \mid g \in \cJ_n\}$ is a family of independent random variables that only depend on the coordinates $x_i$, $i \in \cC(n)$.
\end{itemize}

Since $e$ is the only element in $\cJ$ of length $0$, statements $1_0$ and $2_0$ are trivial. Assume that statements $1_n$ and $2_n$ hold for a given $n$.

Any element in $\F_{1+\kappa}$ of length $n+1$ can be written as a product $gh$ with $|g| = 1$ and $|h| = n$. By the cocycle equality, we have
$$\om(gh,x) = \om(g,h*x) \, \om(h,x) = \om(g, \om(h,x) \cdot x) \, \om(h,x) \; .$$
By statement $1_n$, we know that the map $x \mapsto \om(g,x)$ only depends on the coordinates $x_i$, $i \in \cC(1)$, and that the map $x \mapsto \om(h,x)$ only depends on on the coordinates $x_i$, $i \in \cC(n)$. So, $x \mapsto \om(gh,x)$ only depends on the coordinates $x_i$, $i \in \cC(n)$, and the map
$$x \mapsto (\om(h,x) \cdot x)_{\langle b \rangle k} = r(\om(h,x)) \cdot x_{\langle b \rangle k \om(h,x)} \quad\text{for}\;\; |k| \leq 1 \; .$$
Again by statement $1_n$ these maps only depend on the coordinates $x_i$, $i \in \cC(n+1)$, so that statement $1_{n+1}$ is proven.

Define, for $i = 0,\ldots,\kappa-1$ and $\eps = \pm 1$,
$$\cJ_n^{i,\eps} := \bigl\{g \in \F_{1 + \kappa} \; \big| \; |g| = n \;\;\text{and}\;\; |b_i^{\eps} g| = n+1 \bigr\} \; .$$
It follows that
$$\cJ_{n+1} = \cJ_n \cup \bigcup_{i \in \{0,\ldots,\kappa-1\}, \eps \in \{\pm 1\}} b_i^\eps \, \cJ_n^{i,\eps} \;.$$
Since we assumed that statement $2_n$ holds, in order to prove statement $2_{n+1}$, it suffices to show that
$$\{x \mapsto \rho(b_i^\eps g * x) \mid i=0,\ldots,\kappa-1, \eps = \pm 1, g \in \cJ_n^{i,\eps}\}$$
is a family of independent random variables that only depend on the coordinates $x_i$, $i \in \cC(n+1)$, and that are independent of the coordinates $x_i$, $i \in \cC(n)$.

Note that $\rho(b_i g * x) = \al_{i+1}^{-1} (\rho(b \cdot \theta_i(g*x)))$ while $\rho(b_i^{-1} g * x) = \al_i^{-1}(\rho(b^{-1} \cdot \theta_{i+1}(g*x)))$.
The value of $\rho(b \cdot \theta_i(g*x))$ at $0$ is constantly equal to $i+1$, while the value of $\rho(b^{-1} \cdot \theta_{i+1}(g*x))$ at $0$ is constantly equal to $i$. Therefore we have to prove that
\begin{equation}\label{eq.variables}
\begin{split}
\{x \mapsto \rho & (b \cdot  \theta_i(g*x))_m  \mid i=0,\ldots,\kappa-1 , g \in \cJ_n^{i,+} , m \in \Z - \{0\} \} \\ & \cup
\{x \mapsto \rho(b^{-1} \cdot \theta_{i+1}(g*x))_m \mid i=0,\ldots,\kappa-1 , g \in \cJ_n^{i,-} , m \in \Z - \{0\} \}
\end{split}
\end{equation}
is a family of independent random variables that only depend on the coordinates $x_i$, $i \in \cC(n+1)$, and that are independent of the coordinates $x_i$, $i \in \cC(n)$.

Write
$$\om_i^\eps(g,x) := \begin{cases} b \, \vphi_i(g * x) \, \om(g,x) &\quad\text{if $\eps = 1$,} \\
b^{-1} \, \vphi_{i+1}(g*x) \, \om(g,x) &\quad\text{if $\eps = -1$.}\end{cases}$$
The random variables in \eqref{eq.variables} are precisely equal to
\begin{equation}\label{eq.variablester}
\{x \mapsto r(\om_i^\eps(g,x)) \cdot x_{\langle b \rangle a^m \om_i^\eps(g,x)} \mid i=0,\ldots,\kappa-1 , \eps = \pm 1 , g \in \cJ_n^{i,\eps} , m \in \Z - \{0\} \} \; .
\end{equation}
So we have to prove that \eqref{eq.variablester} is a family of independent random variables that only depend on the coordinates $x_i$, $i \in \cC(n+1)$, and that are independent of the coordinates $x_i$, $i \in \cC(n)$. By statement $1_n$, the maps $x \mapsto \om_i^\eps(g,x)$, and in particular $x \mapsto r(\om_i^\eps(g,x))$, only depend on the coordinates $x_i$, $i \in \cC(n)$. So, we have to prove that
\begin{equation}\label{eq.variablesbis}
\{x \mapsto x_{\langle b \rangle a^m \om_i^\eps(g,x)} \mid i=0,\ldots,\kappa-1 , \eps = \pm 1 , g \in \cJ_n^{i,\eps} , m \in \Z - \{0\} \} \; .
\end{equation}
is a family of independent random variables that only depend on the coordinates $x_i$, $i \in \cC(n+1)$, and that are independent of the coordinates $x_i$, $i \in \cC(n)$.

We apply Lemma \ref{lem.indep} to the countable set $\cC(n+1) - \cC(n)$ and the direct product
$$X_0^{\cC(n)} \times X_0^{\cC(n+1)-\cC(n)} \; .$$
Since the maps $x \mapsto \om_i^\eps(g,x)$ only depend on the coordinates $x_i$, $i \in \cC(n)$, it remains to check that the cosets $\langle b \rangle a^m \om_i^\eps(g,x)$ belong to $\cC(n+1) - \cC(n)$ and that they are distinct for fixed $x \in W_0$ and varying $i \in \{0,\ldots,\kappa-1\}$, $\eps \in \{\pm 1\}$ and $g \in \cJ_n^{i,\eps}$.

Note that $\om(b_i^\eps g,x) \in \langle a \rangle \, \om_i^\eps(g,x)$. Hence,
$$|\om_i^\eps(g,x)| = |\om(b_i^\eps g,x)| = |b_i^\eps g| = n+1$$
because $g \in \cJ_n^{i,\eps}$. Since $|\om(g,x)| = n$ and $|\om_i^\eps(g,x)| = n+1$, it follows from the defining formula of $\om_i^\eps$ that the first letter of $\om_i^\eps(g,x)$ must be $b^\eps$. So the first letter of $a^m \om_i^\eps(g,x)$, $m \neq 0$, is $a^{\pm 1}$. This implies that $\langle b \rangle a^m \om_i^\eps(g,x)$ belongs to $\cC(n+1) - \cC(n)$. It also follows that if
$$\langle b \rangle a^m \om_i^{\eps}(g,x) = \langle b \rangle a^{m'} \om_{i'}^{\eps'}(g',x) \; ,$$
then we must have $m = m'$, $\eps = \eps'$ and $\om_i^\eps(g,x) = \om_{i'}^{\eps'}(g',x)$. Assume $\eps = \eps' = 1$, the other case being analogous. So,
$$\vphi_i(g*x) \, \om(g,x) = \vphi_{i'}(g',x) \, \om(g',x) \; .$$
Applying these elements to $x$, it follows that $\theta_i(g*x) = \theta_{i'}(g'*x)$. Since the ranges of $\theta_i$ and $\theta_{i'}$ are disjoint for $i \neq i'$, it follows that $i=i'$. So, $g*x = g'*x$. Since we have seen above that the action $\F_{1+\kappa} \overset{*}{\actson} W_0$ is essentially free, it follows that $g=g'$.

We have proven that \eqref{eq.variablester} is a family of independent random variables that only depend on the coordinates $x_i$, $i \in \cC(n+1)$, and that are independent of the coordinates $x_i$, $i \in \cC(n)$. So, statement $2_{n+1}$ holds.

To conclude the proof of the lemma, it remains to show that the random variables $x \mapsto \rho(g * x)$, $g \in \F_{1+\kappa}$, generate up to null sets the whole $\si$-algebra of $W_0$. Denote by $\cB_0$ the $\sigma$-algebra on $W_0$ generated by these random variables. By construction, $x \mapsto g*x$ is $\cB_0$-measurable for every $g \in \F_{1+\kappa}$. Since $x \mapsto \rho(x)$ is $\cB_0$-measurable, the formula
$$q(a^n \cdot x) = a^{\eta'(n,\rho(x))} * x$$
shows that $x \mapsto q(a^n \cdot x)$ is $\cB_0$-measurable for every $n \in \Z$. Denote by $\cB_1$ the smallest $\sigma$-algebra on $X$ containing $\cB_0$, containing the subsets $W_0,\ldots,W_{\kappa-1} \subset X$ and making $q : X \recht W_0$ a $\cB_1$-measurable map. Note that the restriction of $\cB_1$ to $W_0$ equals $\cB_0$ and that $\cU \subset X$ is $\cB_1$-measurable if and only if $q(\cU \cap W_i)$ is $\cB_0$-measurable for every $i=0,\ldots,\kappa-1$. It therefore suffices to prove that $\cB_1$ is the whole $\si$-algebra of $X$. By construction, $\rho : X \recht Z$ is $\cB_1$-measurable and by the above, also $x \mapsto a^n \cdot x$ is $\cB_1$-measurable for every $n \in \Z$. If $x \in W_i$, we have that $b \cdot x = \theta_{i+1}^{-1}(b_i * \theta_i(x))$ and it follows that $x \mapsto b \cdot x$ is $\cB_1$-measurable. Hence, $x \mapsto g \cdot x$ is $\cB_1$-measurable for every $g \in \F_2$. Since $\rho$ is $\cB_1$-measurable, it follows that $x \mapsto x_{\langle b \rangle g}$ is $\cB_1$-measurable for every $g \in \F_2$. Hence $\cB_1$ is the entire product $\si$-algebra.
\end{proof}

\subsubsection*{Proof of Lemma \ref{lem.drie}}

\begin{proof}[Proof of Lemma \ref{lem.drie}]
We denote by a dot $\cdot$ the action of $\Gamma$ on $X$. Let $X_1 \subset X$ be a subset of measure $1/\kappa$. We are given a p.m.p.\ action $\Lambda \overset{*}{\actson} X_1$ such that $\Lambda * x = \Gamma \cdot x \cap X_1$ for a.e.\ $x \in X_1$ and such that $\Lambda \actson X_1$ is isomorphic with a $\Lambda$-Bernoulli action. This means that we have a probability space $U$ and a factor map $\pi : X_1 \recht U$
such that the random variables $\{x \mapsto \pi(\lambda * x) \mid \lambda \in \Lambda\}$ are independent, identically distributed and generating the Borel $\sigma$-algebra of $X_1$. Denote by $\om : \Lambda \times X_1 \recht \Gamma$ the $1$-cocycle determined by $\om(\lambda,x) \cdot x = \lambda * x$ for all $\lambda \in \Lambda$ and a.e.\ $x \in X_1$. Put $Y = Y_0^\Gamma$ and define the action $\Lambda \actson X_1 \times Y$ given by
$$\lambda * (x,y) = \om(\lambda,x) \cdot (x,y) = (\lambda * x, \om(\lambda,x) \cdot y) \; .$$
By construction, $\Lambda * (x,y) \subset \Gamma \cdot (x,y) \cap X_1 \times Y$. But also the converse inclusion holds. Indeed, if we have $\gamma \in \Gamma$, $x \in X_1$ and $y \in Y$ such that $\gamma \cdot x \in X_1$, we can take $\lambda \in \Lambda$ such that $\lambda * x = \gamma \cdot x$. Hence $\om(\lambda,x) = \gamma$ and also $\gamma \cdot (x,y) = \lambda * (x,y)$.

It remains to prove that $\Lambda \actson X_1 \times Y$ is isomorphic with a $\Lambda$-Bernoulli action.

By ergodicity of $\Gamma \actson X$, choose a partition (up to measure zero) $X = X_1 \sqcup \cdots \sqcup X_\kappa$ with $\mu(X_i) = 1/\kappa$ and choose measurable maps $\vphi_i : X_1 \recht \Gamma$
such that the formulae $\theta_i(x) = \vphi_i(x) \cdot x$ define measure space isomorphisms $\theta_i : X_1 \recht X_i$. Take $\vphi_1(x) = e$ for all $x \in X_1$. Define the measurable map
$$\rho : X_1 \times Y \recht U \times Y_0^\kappa : \rho(x,y) = (\pi(x),y_{\vphi_1(x)},\ldots,y_{\vphi_\kappa(x)}) \; .$$
We prove that $\rho$ is measure preserving and that the random variables $\{(x,y) \mapsto \rho(\lambda*(x,y)) \mid \lambda \in \Lambda\}$ are independent, identically distributed and generating the Borel $\sigma$-algebra of $X_1 \times Y$.

We first claim that for a.e.\ $x \in X_1$
\begin{equation}\label{eq.F}
\cF := \Bigl(\vphi_i(\lambda * x) \om(\lambda,x)\Bigr)_{\lambda \in \Lambda \;\text{and}\; i=1,\ldots,\kappa}
\end{equation}
is an enumeration of $\Gamma$ without repetitions. Observe that
$$
\vphi_i(\lambda * x) \om(\lambda,x) \cdot x = \theta_i(\lambda * x) \; .
$$
It follows that $\cF \cdot x = \Gamma \cdot x$. Since $\Gamma \actson X$ is essentially free, it follows that $\cF$ enumerates the whole of $\Gamma$. If $\vphi_i(\lambda * x) \om(\lambda,x) = \vphi_j(\lambda' * x) \om(\lambda',x)$, it follows that $\theta_i(\lambda * x) = \theta_j(\lambda' * x)$. For $i \neq j$, the sets $X_i$ and $X_j$ are disjoint. So, $i=j$ and $\lambda * x = \lambda' * x$. Being a Bernoulli action of an infinite group, $\Lambda \overset{*}{\actson} X_1$ is essentially free and we conclude that $\lambda = \lambda'$. This proves the claim.

Since for a.e.\ $x \in X_1$ the elements $\vphi_1(x),\ldots,\vphi_\kappa(x)$ are distinct, it follows from Lemma \ref{lem.indep} that the random variables $(x,y) \mapsto \pi(x)$ and $(x,y) \mapsto y_{\vphi_i(x)}$, $i=1,\ldots,\kappa$, are all independent. Since they are all measure preserving as well, we conclude that
$\rho$ is measure preserving. Note that
$$\rho(\lambda * (x,y)) = \bigl(\pi(\lambda * x),y_{\vphi_1(\lambda * x) \om(\lambda,x)},\ldots,y_{\vphi_\kappa(\lambda * x) \om(\lambda,x)}\bigr) \; .$$
It therefore remains to prove that
$$
\{(x,y) \mapsto \pi(\lambda * x) \mid \lambda \in \Lambda\} \cup \{(x,y) \mapsto y_{\vphi_i(\lambda * x) \om(\lambda,x)} \mid \lambda \in \Lambda, i = 1,\ldots,\kappa\}
$$
is an independent family of random variables that generate, up to null sets, the Borel $\sigma$-algebra of $X_1 \times Y$.
The factor map $\pi$ was chosen in such a way that the random variables $\{ x \mapsto \pi(\lambda * x) \mid \lambda \in \Lambda\}$ are independent and generate, up to null sets, the Borel $\sigma$-algebra of $X_1$. So, we must prove that
\begin{equation}\label{eq.onzefamilie}
\{(x,y) \mapsto y_{\vphi_i(\lambda * x) \om(\lambda,x)} \mid \lambda \in \Lambda, i = 1,\ldots,\kappa\}
\end{equation}
forms a family of independent random variables that are independent of $(x,y) \mapsto x$ and that, together with $(x,y) \mapsto x$, generate up to null sets the Borel $\sigma$-algebra of $X_1 \times Y$. We apply Lemma \ref{lem.indep} to the countable set $\Gamma$, the direct product $X_1 \times Y_0^\Gamma$ and the family of maps $X_1 \recht \Gamma : x \mapsto \vphi_i(\lambda * x) \om(\lambda,x)$ indexed by $\lambda \in \Lambda, i = 1,\ldots,\kappa$. Since for a.e.\ $x \in X_1$, the set $\cF$ in \eqref{eq.F} is an enumeration of $\Gamma$, it follows from Lemma \ref{lem.indep} that \eqref{eq.onzefamilie} is indeed a family of independent random variables that are moreover independent of $(x,y) \mapsto x$.

Denote by $\cB_1$ the smallest $\sigma$-algebra on $X_1 \times Y$ such that the map $(x,y) \mapsto x$ and the random variables in \eqref{eq.onzefamilie} are measurable. It remains to prove that, up to null sets, $\cB_1$ is the Borel $\sigma$-algebra of $X_1 \times Y$. So, it remains to prove that for all $g \in \Gamma$, the random variables $(x,y) \mapsto y_g$ are $\cB_1$-measurable.
Put $\cJ = \{1,\ldots,\kappa\} \times \Lambda$ and define the Borel map $\eta : \cJ \times X_1 \recht \Gamma$ given by $\eta(i,\lambda,x) := \vphi_i(\lambda * x) \om(\lambda,x)$. Since for a.e.\ $x \in X_1$, the family $\cF$ in \eqref{eq.F} is an enumeration of $\Gamma$, we can take a Borel map $\gamma : \Gamma \times X_1 \recht \cJ$ such that $\eta(\gamma(g,x),x) = g$ for all $g \in \Gamma$ and a.e.\ $x \in X_1$. By the definition of $\cB_1$ and $\eta$, we know that the map
\begin{equation}\label{eq.eenmap}
\cJ \times X_1 \times Y \recht Y_0 : (j,x,y) \mapsto y_{\eta(j,x)}
\end{equation}
is $\cB_1$-measurable. Fix $g \in \Gamma$. Since $(x,y) \mapsto x$ is $\cB_1$-measurable, also $(x,y) \mapsto (\gamma(g,x),x,y)$ is $\cB_1$-measurable. The composition with the map in \eqref{eq.eenmap} yields $(x,y) \mapsto y_g$ a.e. So $(x,y) \mapsto y_g$ is $\cB_1$-measurable. This concludes the proof of the lemma.
\end{proof}

\section{Isomorphisms of factors of Bernoulli actions of free products}

Before proving Theorem \ref{thm.B}, we need the following elementary lemma.

\begin{lemma}\label{lem.factor}
Let $\Gamma,\Lambda$ be countable groups and $K$ a nontrivial second countable compact group equipped with its normalized Haar measure. Consider the action $(\Gamma * \Lambda) \times K \actson X := K^{\Gamma \backslash \Gamma * \Lambda}$ where $\Gamma * \Lambda$ shifts the indices and $K$ acts by diagonal left translation. The resulting factor action $\Gamma * \Lambda \actson X/K$ is isomorphic with the co-induced action of $\Lambda \actson K^\Lambda / K$ to $\Gamma * \Lambda$.
\end{lemma}
\begin{proof}
Define the factor map $\rho : K^{\Gamma \backslash \Gamma * \Lambda} \recht K^\Lambda$ given by $\rho(x)_\lambda = x_{\Gamma \lambda}$. Note that $\rho$ is $(\Lambda \times K)$-equivariant. Denote $X := K^{\Gamma \backslash \Gamma * \Lambda}$ and denote by $x \mapsto \xb$ the factor map of $X$ onto $X/K$. So we get the $\Lambda$-equivariant factor map $\rhob : X/K \recht K^\Lambda/K : \rhob(\xb) = \overline{\rho(x)}$. We prove that $\Gamma * \Lambda \actson X/K$ together with $\rhob$ satisfies the abstract characterization of the co-induced action of $\Lambda \actson K^\Lambda/K$ to $\Gamma * \Lambda$.

For $g \in \Gamma * \Lambda$, denote by $|g|$ the number of letters from $\Gamma - \{e\}$ appearing in a reduced expression for $g$. Define the subsets $I_n \subset \Gamma * \Lambda$ given by $I_0 := \{e\}$ and
$$I_n := \bigl\{g \in \Gamma * \Lambda \;\big| \; |g| = n \;\;\text{and the reduced expression of $g$ starts with a letter from $\Gamma - \{e\}$} \; \bigr\} \; .$$
Note that $\bigcup_{n=0}^\infty I_n$ is a right transversal for $\Lambda < \Gamma * \Lambda$. So we have to prove that
\begin{equation}\label{eq.mynicefam}
\{\xb \mapsto \rhob(g \cdot \xb) \mid n \in \N , g \in I_n \}
\end{equation}
is a family of independent random variables that generate, up to null sets, the whole $\sigma$-algebra of $X/K$.

For $i \in \Gamma \backslash \Gamma * \Lambda$, we write $|i|=n$ if $i = \Gamma g$, where $|g|=n$ and the reduced expression for $g$ starts with a letter from $\Lambda - \{e\}$. For every $\lambda \in \Lambda - \{e\}$, define $\theta_\lambda : K^\Lambda/K \recht K : \theta_\lambda(\xb) = x_e^{-1} x_\lambda$. Observe that for all $g \in I_n$ and $\lambda \in \Lambda-\{e\}$, we have
\begin{equation}\label{eq.consequence}
\theta_\lambda(\rhob(g \cdot \xb)) = x_{\Gamma g}^{-1} \, x_{\Gamma \lambda g} \; .
\end{equation}
Since $g \in I_n$ starts with a letter from $\Gamma-\{e\}$, we have $|\Gamma \lambda g| = |g| = n$, while $|\Gamma g| = n-1$.
Write $\cI_n := \{i \in \Gamma \backslash \Gamma * \Lambda \mid |i| \leq n\}$. We apply Lemma \ref{lem.indep} to the countable set $\cI_n - \cI_{n-1}$, the direct product
$$Z := K^{\cI_{n-1}} \times K^{\cI_n - \cI_{n-1}}$$
and the family of maps $\om_{g,\lambda} : Z \recht K \times (\cI_n - \cI_{n-1})$, indexed by $g \in I_n, \lambda \in \Lambda - \{e\}$, only depending on the coordinates $x_i$, $i \in \cI_{n-1}$, and given by
$$\om_{g,\lambda} : x \mapsto (x_{\Gamma g}^{-1}, \Gamma \lambda g) \; .$$
Since the elements $\Gamma \lambda g$, for $g \in I_n,\lambda \in \Lambda - \{e\}$, enumerate $\cI_n - \cI_{n-1}$, it follows from Lemma \ref{lem.indep} that the random variables
$$\{X \recht K : x \mapsto x_{\Gamma g}^{-1} x_{\Gamma \lambda g} \mid g \in I_n , \lambda \in \Lambda - \{e\} \; \}$$
are independent, only depend on the coordinates $x_i$, $|i| \leq n$, and are independent of the coordinates $x_i$, $|i| \leq n-1$. In combination with \eqref{eq.consequence}, it follows that \eqref{eq.mynicefam} is indeed a family of independent random variables.

Denote by $\cB_0$ the smallest $\sigma$-algebra on $X/K$ for which all the functions $\xb \mapsto \rhob(g \cdot \xb)$, $g \in \Gamma * \Lambda$, are $\cB_0$-measurable.
Formula \eqref{eq.consequence} and an induction on $n$ show that $\xb \mapsto x_{\Gamma e}^{-1} \, x_i$ is $\cB_0$-measurable for every $i \in \Gamma \backslash \Gamma * \Lambda$ with $|i|\leq n$. Hence, $\cB_0$ is the entire $\sigma$-algebra on $X/K$.
\end{proof}

Theorem \ref{thm.B} will be an immediate corollary of the following general result.

\begin{theorem}\label{thm.stability}
Let $\Gamma_i$, $i=0,1$, be countable groups and $K$ a nontrivial second countable compact group equipped with its normalized Haar measure. Assume that $\Gamma_i \actson K^{\Gamma_i} / K$ is isomorphic with the Bernoulli action $\Gamma_i \actson Y_i^{\Gamma_i}$ with base space $(Y_i,\mu_i)$. Write $G := \Gamma_0 * \Gamma_1$. Then $G \actson K^{G}/K$ is isomorphic with the Bernoulli action $G \actson (Y_0 \times Y_1)^{G}$ with base space $Y_0 \times Y_1$.
\end{theorem}
\begin{proof}
Put $A := K^{\Gamma_0}$ and denote by $\al$ the action $\Gamma_0 \times K \overset{\al}{\actson} A$ where $\Gamma_0$ shifts the indices and $K$ acts by diagonal left translation. Put $B := Y_0^{\Gamma_0} \times K$ and denote by $\beta$ the action $\Gamma_0 \times K \overset{\beta}{\actson} B$ where $\Gamma_0$ only acts on the factor $Y_0^{\Gamma_0}$ in a Bernoulli way and $K$ only acts on the factor $K$ by translation. Our assumptions say that $\Gamma_0 \actson A/K$ and $\Gamma_0 \actson B/K$ are isomorphic actions. We apply Theorem \ref{thm.coinduced} to these two actions of $\Gamma_0$.

So denote $G = \Gamma_0 * \Gamma_1$ and denote by $G \actson \Atil$, resp.\ $G \actson \Btil$, the co-induced actions of $\Gamma_0 \actson A$, resp.\ $\Gamma_0 \actson B$, to $G$. Note that these actions come together with natural actions $K \actson \Atil$ and $K \actson \Btil$ that commute with $G$-actions. By Theorem \ref{thm.coinduced}, the actions $G \actson \Atil / K$ and $G \actson \Btil / K$ are isomorphic.

We now identify the actions $G \times K \actson \Atil$ and $G \times K \actson \Btil$ with the following known actions. First, the action $G \times K \actson \Atil$ is canonically isomorphic with $G \times K \actson K^G$ where $G$ acts in a Bernoulli way and $K$ acts by diagonal left translation. Secondly, using Remark \ref{rem.bete}.3, the action $G \times K \actson \Btil$ is isomorphic with the action $G \times K \actson Y_0^{G} \times K^{\Gamma_0 \backslash G}$ where $G$ acts diagonally in a Bernoulli way and $K$ only acts on the second factor by diagonal left translation. In combination with the previous paragraph, it follows that the action $G \actson K^{G}/K$ is isomorphic with the diagonal action $G \actson Y_0^G \times (K^{\Gamma_0 \backslash G})/K$.

From Lemma \ref{lem.factor}, we know that $G \actson (K^{\Gamma_0 \backslash G})/K$ is isomorphic with the co-induced action of $\Gamma_1 \actson K^{\Gamma_1}/K$ to $G$. Since we assumed that $\Gamma_1 \actson K^{\Gamma_1}/K$ is isomorphic with the Bernoulli action $\Gamma_1 \actson Y_1^{\Gamma_1}$, it follows that
$G \actson (K^{\Gamma_0 \backslash G})/K$ is isomorphic with the Bernoulli action $G \actson Y_1^G$. In combination with the previous paragraph, it follows that $G \actson K^G/K$ is isomorphic with the Bernoulli action $G \actson (Y_0 \times Y_1)^G$.
\end{proof}

\subsection*{Proof of Theorem \ref{thm.B}}

\begin{proof}[Proof of Theorem \ref{thm.B}]
Since the action $\Lambda_i \actson K^{\Lambda_i} / K$ arises as the factor of a Bernoulli action and $\Lambda_i$ is amenable, it follows from \cite{OW86} that $\Lambda_i \actson K^{\Lambda_i} / K$ is isomorphic with a Bernoulli action $\Lambda_i \actson Y_i^{\Lambda_i}$. Repeatedly applying Theorem \ref{thm.stability}, it follows that $\Gamma \actson K^{\Gamma}/K$ is isomorphic with the Bernoulli action $\Gamma \actson (Y_1 \times \cdots \times Y_n)^\Gamma$.

The special case $\Gamma = \F_n$ is a very easy generalization of \cite[Appendix C.(b)]{OW86}. Denote by $x \mapsto \xb$ the quotient map from $K^{\F_n}$ to $K^{\F_n}/K$. Denote by $a_1,\ldots,a_n$ the free generators of $\F_n$. Define the measurable map
$$
\theta : K^{\F_n}/K \recht (K \times \cdots \times K)^{\F_n} : \theta(\xb)_g = (x_g^{-1} \, x_{a_1 g},\ldots, x_g^{-1} \, x_{a_n g}) \; .
$$
We shall prove that $\theta$ is an isomorphism between $\F_n \actson K^{\F_n}/K$ and $\F_n \actson (K \times \cdots \times K)^{\F_n}$. First note that $\theta$ is indeed $\F_n$-equivariant. It remains to prove that
\begin{equation}\label{eq.myfamily}
\{\xb \mapsto x_g^{-1} \, x_{a_i g} \mid i=1,\ldots,n \; , \; g \in \F_n \}
\end{equation}
is a family of independent random variables on $K^{\F_n}/K$ that generate up to null sets the whole $\sigma$-algebra of $K^{\F_n}/K$. Denote by $|g|$ the word length of an element $g \in \F_n$. Define for $i \in \{1,\ldots,n\}$, $\eps = \pm 1$ and $k \in \N$, the subsets $I^{i,\eps}_k \subset \F_n$ given by
$$I^{i,\eps}_k := \bigl\{g \in \F_n \;\big|\; |g| = k \; , \; |a_i^{\eps} g|=k+1 \bigr\} \; .$$
If $|g| = k$ and $|a_i g| = k-1$, we compose the random variable $\xb \mapsto x_g^{-1} \, x_{a_i g}$ by the map $K \recht K : y \mapsto y^{-1}$ and observe that $a_i g \in I^{i,-1}_{k-1}$. So we need to prove that
\begin{equation}\label{eq.secondfamily}
\{\xb \mapsto x_g^{-1} \, x_{a_i^\eps g} \mid i = 1,\ldots,n \; , \; \eps = \pm 1 \; , \; k \in \N \; , \; g \in I^{i,\eps}_k \}
\end{equation}
is a family of independent random variables that generate up to null sets the whole $\sigma$-algebra of $K^{\F_n}/K$.

Write $\cI_k = \{g \in \F_n \mid |g| \leq k\}$ and fix $k \in \N$. We apply Lemma \ref{lem.indep} to the countable set $\cI_{k+1}- \cI_k$, the direct product
$$Z := K^{\cI_k} \times K^{\cI_{k+1} - \cI_k}$$
and the family of maps $\om_{i,\eps,g} : Z \recht K \times (\cI_{k+1} - \cI_k)$ indexed by the set
$$\cF := \{(i,\eps,g) \mid i = 1,\ldots,n \; , \; \eps = \pm 1 \; , \; g \in I^{i,\eps}_k \} \; ,$$
only depending on the coordinates $x_i$, $i \in \cI_k$, and given by
$$\om_{i,\eps,g} : x \mapsto (x_g^{-1}, a_i^\eps g) \; .$$
Since the elements $a_i^\eps g$ with $(i,\eps,g) \in \cF$ enumerate $\cI_{k+1} - \cI_k$, it follows from Lemma \ref{lem.indep} that
$\{x \mapsto x_g^{-1} \, x_{a_i^\eps g} \mid i=1,\ldots,n \; , \; \eps = \pm 1 \; , \; g \in I^{i,\eps}_k\}$ is a family of independent random variables that are independent of the coordinates $x_h$, $|h| \leq k$. By construction, these random variables only depend on the coordinates $x_h$, $|h|\leq k+1$. This being proven for all $k \in \N$, it follows that \eqref{eq.secondfamily} is a family of independent random variables. Hence the same is true for \eqref{eq.myfamily}. These random variables can be easily seen to generate up to null sets the whole $\sigma$-algebra of $K^{\F_n}/K$.
\end{proof}

\section*{Appendix: essentially free actions of locally compact groups}

A p.m.p.\ action of a second countable locally compact group $G$ on a standard probability space $(X,\mu)$ is an action of the group $G$ on the set $X$ such that $G \times X \recht X : (g,x) \mapsto g \cdot x$ is a Borel map and such that for all $g \in G$ and all Borel sets $A \subset X$, we have $\mu(g \cdot A) = \mu(A)$.

For every $x \in X$, we define the subgroup $\Stab x$ of $G$ given by $\Stab x = \{g \in G \mid g \cdot x = x\}$. For the sake of completeness, we give a proof for the following folklore lemma.

\begin{lemma}\label{lem.compact-free}
Let $G \actson (X,\mu)$ be a p.m.p.\ action of a second countable locally compact group $G$ on a standard probability space $(X,\mu)$, as above.
\begin{enumerate}
\item The set $X_0 := \{x \in X \mid \Stab x = \{e\} \; \}$ is a $G$-invariant Borel subset of $X$.
\item Assume that $\mu(X_0) = 1$ and that $G$ is compact. Denote by $m$ the normalized Haar measure on $G$. There exists a standard probability space $(Y_0,\eta)$ and a bijective Borel isomorphism
$\theta : G \times Y_0 \recht X_0$ such that $\theta(gh,y) = g \cdot \theta(h,y)$ for all $g,h \in G$, $y \in Y_0$, and such that $\theta_*(m \times \eta) = \mu$.
\end{enumerate}
\end{lemma}

A p.m.p.\ action $G \actson (X,\mu)$ is called essentially free if the Borel set $\{x \in X \mid \Stab x = \{e\}\}$ has measure $1$.

\begin{proof}
By \cite[Theorem 3.2]{Va62}, there exists a continuous action of $G$ on a Polish space $Y$ and an injective Borel map $\psi : X \recht Y$ satisfying $\psi(g \cdot x) = g \cdot \psi(x)$ for all $g \in G$ and $x \in X$. Since $\psi$ is injective, $\psi(X)$ is a Borel subset of $Y$ and $\psi$ is a Borel isomorphism of $X$ onto $\psi(X)$ (see e.g.\ \cite[Theorem 15.1]{Ke95}). So, we actually view $X$ as a $G$-invariant Borel subset of $Y$.

To prove 1, fix a sequence of compact subsets $K_n \subset G - \{e\}$ such that $G-\{e\} = \bigcup_{n=1}^\infty K_n$. Also fix a metric $d$ on $Y$ that induces the topology on $Y$.
Define
$$f_n : X \recht \R : f_n(x) = \min_{g \in K_n} d(g \cdot x,x) \; .$$
Whenever $\cF_n \subset K_n$ is a countable dense subset, we have $f_n(x) = \inf_{g \in \cF_n} d(g \cdot x,x)$, so that $f_n$ is Borel. Since $\Stab x = \{e\}$ if and only if $f_n(x) > 0$ for all $n$, statement 1 follows.

To prove 2, assume that $\mu(X_0) = 1$ and that $G$ is compact. Since $G$ acts continuously on $Y$ and $G$ is compact, all orbits $G \cdot y$ are closed. By \cite[Theorem 12.16]{Ke95}, we can choose a Borel subset $Y_1 \subset Y$ such that $Y_1 \cap G \cdot y$ is a singleton for every $y \in Y$. Define $Y_0 := Y_1 \cap X_0$. By construction, the map
$$\theta : G \times Y_0 \recht X_0 : \theta(g,y) = g \cdot y$$
is Borel, bijective and satisfies $\theta(gh,y) = g \cdot \theta(h,y)$ for all $g,h \in G$ and $y \in Y_0$. Then also $\theta^{-1}$ is Borel (see e.g.\ \cite[Theorem 15.1]{Ke95}). The formula $\eta_0 := (\theta^{-1})_*(\mu)$ yields a $G$-invariant probability measure on $G \times Y_0$. Defining the probability measure $\eta$ on $Y_0$ as the push forward of $\eta_0$ under the quotient map $(g,y) \mapsto y$, the $G$-invariance of $\eta_0$ together with the Fubini theorem, imply that $\eta_0 = m \times \eta$.
\end{proof}


\begin{thebibliography}{Gr10a}\setlength{\itemsep}{-1mm} \setlength{\parsep}{0mm} \small

\bibitem[Bo08]{Bo08} L. Bowen, A new measure conjugacy invariant for actions of free groups. {\it Ann. of Math.} {\bf 171} (2010), 1387-1400.

\bibitem[Bo09a]{Bo09a} L. Bowen, Orbit equivalence, coinduced actions and free products. {\it Groups Geom. Dyn.} {\bf 5} (2011), 1-15.

\bibitem[Bo09b]{Bo09b} L. Bowen, Stable orbit equivalence of Bernoulli shifts over free groups. {\it Groups Geom. Dyn.} {\bf 5} (2011), 17-38.

\bibitem[Dy58]{Dy58} H.A. Dye, On groups of measure preserving transformations, I.
{\it Amer. J. Math.} {\bf 81} (1959), 119-159.

\bibitem[Ep07]{Ep07} I. Epstein, Orbit inequivalent actions of non-amenable groups. {\it Preprint.} {\tt arXiv:0707.4215}

\bibitem[Fu09]{Fu09} {A. Furman}, A survey of measured group theory.
In {\it Geometry, rigidity, and group actions,} Eds. B. Farb and D. Fisher. The University of Chicago Press, 2011, pp.\ 296-374.

\bibitem[Ga10]{Ga10} {D. Gaboriau}, Orbit equivalence and measured group theory. In {\it Proceedings of the International Congress of Mathematicians (Hyderabad, India, 2010),} Vol.\ III, Hindustan Book Agency, 2010, pp.\ 1501-1527.

\bibitem[GL07]{GL07} {D. Gaboriau and R. Lyons}, A measurable-group-theoretic solution to von Neumann's problem. {\it Invent. Math.} {\bf 177} (2009), 533-540.

\bibitem[GP03]{GP03} {D. Gaboriau and S. Popa}, An uncountable family of nonorbit equivalent actions of $\F_n$.
{\it J. Amer. Math. Soc.} {\bf 18} (2005), 547-559.

\bibitem[Ho11]{Ho11} {C. Houdayer}, Invariant percolation and measured theory of nonamenable groups (after Gaboriau-Lyons, Ioana, Epstein). {\it S\'{e}minaire Bourbaki, exp.\ 1039,} to appear in {\it Ast\'{e}risque.} {\tt arXiv:1106.5337}

\bibitem[Io06]{Io06} {A. Ioana}, Orbit inequivalent actions for groups containing a copy of $\F_2$. {\it Invent. Math.} {\bf 185} (2011), 55-73.

\bibitem[Ke95]{Ke95} {A.S. Kechris}, Classical descriptive set theory. {\it Graduate Texts in Mathematics} {\bf 156}, Springer-Verlag, New York, 1995.

\bibitem[OW79]{OW79} D. Ornstein and B. Weiss, Ergodic theory of amenable group actions, I. {\it Bull. Amer. Math. Soc. (N.S.)} {\bf 2} (1980), 161-164.

\bibitem[OW86]{OW86} D. Ornstein and B. Weiss, Entropy and isomorphism theorems for actions of amenable groups. {\it J. Analyse Math.} {\bf 48} (1987), 1-141.

\bibitem[PS09]{PS09} {J. Peterson and T. Sinclair,} On cocycle superrigidity for Gaussian actions. {\it Erg. Th. Dyn. Sys.} {\bf 32} (2012), 249-272.

\bibitem[Po05]{Po05} {S. Popa}, Cocycle and orbit equivalence superrigidity for malleable actions of $w$-rigid groups. {\it Invent. Math.} {\bf 170} (2007), 243-295.

\bibitem[Po06]{Po06} {S. Popa}, On the superrigidity of malleable actions with spectral gap. {\it J. Amer. Math. Soc.} {\bf 21} (2008), 981-1000.

\bibitem[PV06]{PV06} {S. Popa and S. Vaes}, Strong rigidity of generalized Bernoulli actions and computations of their symmetry groups. {\it Adv. Math.} {\bf 217} (2008), 833-872.

\bibitem[Sh05]{Sh05} {Y. Shalom}, Measurable group theory. In {\it European Congress of Mathematics,}  European Mathematical Society Publishing House, 2005, pp.\ 391-423.

\bibitem[Si55]{Si55} {I.M. Singer}, Automorphisms of finite factors. {\it Amer. J. Math.} {\bf 77} (1955), 117-133.

\bibitem[Va62]{Va62} {V.S. Varadarajan}, Groups of automorphisms of Borel spaces. {\it Trans. Amer. Math. Soc.} {\bf 109} (1963), 191-220.
\end{thebibliography}
\end{document}